\documentclass[12pt]{amsart}
\usepackage{amssymb, enumerate, palatino, amsfonts, latexsym, amscd}
\usepackage[mathscr]{euscript}
\usepackage{xy} \xyoption{all}
\DeclareMathOperator{\Cu}{Cu}
\DeclareMathOperator{\PZ}{PZ}
\DeclareMathOperator{\tr}{{tr}}

\newcommand{\K}{{\mathcal{K}}}
\newcommand{\B}{{\mathbb{B}}}
\newcommand{\R}{{\mathbb{R}}}
\newcommand{\N}{{\mathbb{N}}}

\newcommand{\Cs}{{$C^*$-algebra}}
\newcommand{\ep}{{\varepsilon}}
\newcommand{\Proj}{P}

\newcommand{\cO}{{\mathcal{O}}}
\newcommand{\cM}{{\mathcal{M}}}
\newcommand{\ssubset}{\ensuremath{\subset \hskip-5pt \subset}}

\newcommand{\uloopr}[1]{\ar@'{@+{[0,0]+(-4,5)}@+{[0,0]+(0,10)}@+{[0,0] +(4,5)}}^{#1}}
\newcommand{\uloopd}[1]{\ar@'{@+{[0,0]+(5,4)}@+{[0,0]+(10,0)}@+{[0,0]+ (5,-4)}}^{#1}}
\newcommand{\dloopr}[1]{\ar@'{@+{[0,0]+(-4,-5)}@+{[0,0]+(0,-10)}@+{[0, 0]+(4,-5)}}_{#1}}
\newcommand{\dloopd}[1]{\ar@'{@+{[0,0]+(-5,4)}@+{[0,0]+(-10,0)}@+{[0,0 ]+(-5,-4)}}_{#1}}
\newcommand{\luloop}[1]{\ar@'{@+{[0,0]+(-8,2)}@+{[0,0]+(-10,10)}@+{[0, 0]+(2,2)}}^{#1}}

\newtheorem{lem}{Lemma}[section]
\newtheorem{corol}[lem]{Corollary}
\newtheorem{theor}[lem]{Theorem}
\newtheorem{prop}[lem]{Proposition}
\newtheorem{defi}[lem]{Definition}
\theoremstyle{definition}

\newtheorem{pargr}[lem]{}
\newtheorem{rema}[lem]{Remark}

\begin{document}

\title[The Cuntz semigroup and comparison of open projections]{The Cuntz semigroup and comparison of open projections}%
\author{Eduard Ortega, Mikael R\o rdam and Hannes Thiel}

\address{Department of Mathematical Sciences, NTNU, NO-7491 Trondheim, Norway}
\email{eduardo.ortega@math.ntnu.no}
\address{Department of Mathematical Sciences, University of Copenhagen, Universitetsparken 5, DK-2100, Copenhagen \O, Denmark} \email{rordam@math.ku.dk}
\address{Department of Mathematical Sciences, University of Copenhagen, Universitetsparken 5, DK-2100, Copenhagen \O, Denmark} \email{thiel@math.ku.dk}

\thanks{This research was supported by the NordForsk Research Network
  ``Operator Algebras and Dynamics'' (grant \#11580).
  The first named author was
  partially supported by the Research Council of Norway (project 191195/V30), by  MEC-DGESIC (Spain) through Project
  MTM2008-06201-C02-01/MTM, by the Consolider Ingenio
``Mathematica" project CSD2006-32 by the MEC, and by 2009 SGR 1389 grant of the Comissionat per Universitats i Recerca de la Generalitat de
Catalunya. The second named author was supported by grants from the Danish National Research Foundation and the Danish Natural Science Research Council (FNU). The third named author was partially supported by the Marie Curie Research Training Network EU-NCG and by the Danish National Research Foundation}
\date{\today}

\begin{abstract}
We show that a number of naturally occurring comparison relations on positive elements in a \Cs{} are equivalent to natural comparison properties of their corresponding open projections in the bidual of the \Cs.
In particular we show that Cuntz comparison of positive elements corresponds to a comparison relation on open projections, that we call Cuntz comparison, and which is defined in terms of---and is weaker than---a comparison notion defined by Peligrad and Zsid\'{o}. The latter corresponds to a well-known comparison relation on positive elements defined by Blackadar. We show that Murray-von Neumann comparison of open projections corresponds to tracial comparison of the corresponding positive elements of the \Cs. We use these findings to give a new picture of the Cuntz semigroup.
\end{abstract}

\maketitle

%########################################################################################################################
%########################################################################################################################
\section{Introduction}

\noindent
    There is a well-known bijective correspondence between hereditary sub-\Cs s of a \Cs{} and open projections in its bidual. Thus to every positive element $a$ in a \Cs{} $A$ one can associate the open projection $p_a$ in $A^{**}$ corresponding to the hereditary sub-\Cs{} $A_a = \overline{aAa}$.
 Any comparison relation between positive elements in a \Cs{} that is invariant under the relation $a \cong b$, defined by $a \cong b \Leftrightarrow A_a = A_b$, can in this way be translated into a comparison relation between open projections in the bidual. Vice versa, any comparison relation between open projections corresponds to a comparison relation (which respects $\cong$) on positive elements of the underlying \Cs.

    Peligrad and Zsid\'{o} defined in \cite{PelZsi} an equivalence relation (and also a sub-equivalence relation) on open projections in the bidual of a \Cs{} as Murray-von Neumann equivalence with the extra assumption that the partial isometry that implements the equivalence gives an isomorphism between the corresponding hereditary sub-\Cs s of the given \Cs. Very recently, Lin, \cite{Lin10}, noted that the Peligrad-Zsid\'{o} \mbox{(sub-)}equivalence of open projections corresponds to a comparison relation of positive elements considered by Blackadar in \cite{Bla}.

    The Blackadar comparison of positive elements is stronger than the Cuntz comparison relation of positive elements that is used to define the Cuntz semigroup of a \Cs. The Cuntz semigroup has recently come to play an influential role in the classification of \Cs s. We show that Cuntz comparison of positive elements corresponds to a natural relation on open projections, that we also call Cuntz comparison. It is defined in terms of---and is weaker than---the Peligrad-Zsid\'{o} comparison.
    It follows from results of Coward, Elliott, and Ivanescu, \cite{CowEllIva}, and from our results, that the Blackadar comparison is equivalent to the Cuntz comparison of positive elements when the \Cs{} is separable and has stable rank one, and consequently that the Peligrad-Zsid\'{o} comparison is equivalent to our notion of Cuntz comparison of open projections in this case.

    The best known and most natural comparison relation for projections in a von Neumann algebra is the one introduced by Murray and von Neumann. It is weaker than the Cuntz and the Peligrad-Zsid\'{o} comparison relations. We show that Murray-von Neumann \mbox{(sub-)}equivalence of open projections in the bidual in the separable case is equivalent to tracial comparison of the corresponding positive elements of the \Cs.
    The tracial comparison is defined in terms of dimension functions arising from lower semicontinuous tracial weights on the \Cs.
    The proof of this equivalence builds on two results on von Neumann algebras that may have independent interest, and which probably are known to experts:
    One says that Murray-von Neumann comparison of projections in any von Neumann algebra which is not too big (in the sense of Tomiyama---see Section \ref{sec5} for details) is completely determined by normal tracial weights on the von Neumann algebra.
    The other results states that every lower semicontinuous tracial weight on a \Cs{} extends (not necessarily uniquely) to a normal tracial weight on the bidual of the \Cs.

    We use results of Elliott, Robert, and Santiago, \cite{EllRobSan}, to show that tracial comparison of positive elements in a \Cs{} is equivalent to Cuntz comparison if the \Cs{} is separable and exact, its Cuntz semigroup is weakly unperforated, and the involved positive elements are purely non-compact.

    We also relate the comparison of positive elements and of open projections to comparison of the associated right Hilbert $A$-modules.
    The Hilbert $A$-module corresponding to a positive element $a$ in $A$ is the right ideal $\overline{aA}$.
    We show that Blackadar equivalence of positive elements is equivalent to isomorphism of the corresponding Hilbert $A$-modules, and we recall that Cuntz comparison of positive elements is equivalent to a the notion of Cuntz comparison of the corresponding Hilbert $A$-modules introduced in \cite{CowEllIva}.

%########################################################################################################################
%########################################################################################################################
\section{Comparison of positive elements in a \Cs}\label{sec2}

\noindent  We remind the reader about some, mostly well-known, notions of comparison of positive elements in a \Cs. If $a$ is a positive element in a \Cs{} $A$, then let $A_a$ denote the hereditary sub-\Cs{} generated by $a$, i.e.,  $A_a=\overline{aAa}$. The \emph{Pedersen equivalence relation} on positive elements in a \Cs{} $A$ is defined by $a \sim b$ if $a = x^*x$ and $b = xx^*$ for some $x \in A$, where $a,b \in A^+$, and it was shown by Pedersen, that this indeed defines an equivalence relation.
Write $a\cong b$ if $A_a=A_b$. The equivalence generated by these two relations was considered by Blackadar in \cite[Definition 6.1.2]{Bla2}:

\begin{defi}[Blackadar comparison] \label{def:blackadar-comparison}
Let $a$ and $b$ be positive elements in a \Cs{} $A$. Write $a \sim_s b$ if there exists $x \in A$ such that $a \cong x^*x$ and $b  \cong xx^*$, and write $a \precsim_s b$ if there exists $a'\in A_b^+$ with $a\sim_s a'$.
\end{defi}

\noindent
(It follows from Lemma \ref{lem:TFAE:equivalence} below that $\sim_s$ is an equivalence relation.)
Note that $\precsim_s$ is not an order relation on $A^+/\! \!\sim_s$ since in general $a\precsim_s b\precsim_s a$ does not imply $a\sim_s b$ (see \cite[Theorem 9]{Lin1990}). If $p$ and $q$ are projections, then $p \sim_s q$ agrees with the usual notion of equivalence of projections defined by Murray and von Neumann, denoted by $p \sim q$.

The relation defining the Cuntz semigroup that currently is of importance in the classification program for \Cs s is defined as follows:

\begin{defi}[Cuntz comparison of positive elements] \label{def:cuntz-comparison}
Let $a$ and $b$ be positive elements in a \Cs{} $A$. Write $a \precsim b$ if there exists a sequence
$\{x_n\}$ in $A$ such that $x_n^*bx_n \to a$. Write $a \approx b$ if $a\precsim b$ and $b\precsim a$.
\end{defi}

\begin{pargr}[The Cuntz semigroup]
    Let us briefly remind the reader about the ordered Cuntz semigroup $W(A)$ associated to a \Cs{} $A$.
    Let $M_{\infty}(A)^+$ denote the disjoint union $\bigcup_{n=1}^\infty M_n(A)^+$.
    For $a\in M_n(A)^+$ and $b\in M_m(A)^+$ set $a\oplus b=\text{diag }(a,b)\in M_{n+m}(A)^+$, and write $a\precsim b$ if there exists a sequence $\{x_k\}$ in $M_{m,n}(A)$ such that $x_k^*bx_k\longrightarrow a$.
    Write $a\approx b$ if $a\precsim b$ and $b\precsim a$.
    Put $W(A)=M_{\infty}(A)^+ / \!\! \approx$, and let $\langle a\rangle\in W(A)$ be the equivalence class containing $a$.
    Let us denote by $\Cu(A)$ the completed Cuntz semigroup, i.e., $\Cu(A):=W(A\otimes\K)$.
\end{pargr}

\noindent
    Lastly we define comparison by traces.
    We shall here denote by $T(A)$ the set of (norm) lower semicontinuous tracial weights on a \Cs{} $A$.
    We remind the reader that a tracial weight on $A$ is an additive function $\tau \colon A^+ \to [0,\infty]$  satisfying $\tau(\lambda a) = \lambda \tau(a)$ and $\tau(x^*x) = \tau(xx^*)$ for all $a \in A^+$, $x \in A$, and $\lambda \in \R^+$. That $\tau$ is lower semicontinuous means that $\tau(a) = \lim \tau(a_i)$ whenever $\{a_i\}$ is a norm-convergent increasing sequence (or net) with limit $a$.
    Each $\tau \in T(A)$ gives rise to a lower semicontinuous dimension function $d_\tau \colon A^+ \to [0,\infty]$ given by $d_\tau(a) = \sup_{\ep>0} \tau(f_\ep(a))$, where $f_\ep \colon \R^+ \to \R^+$ is the continuous function that is $0$ on $0$, $1$ on $[\ep,\infty)$, and linear on $[0,\ep]$.
    Any dimension function gives rise to an additive order preserving state on the Cuntz semigroup, and in particular it preserves the Cuntz relation $\precsim$.

\begin{defi}[Comparison by traces]
\label{def:trace-comparison}
    Let $a$ and $b$ be positive elements in a \Cs{} $A$.
    Write $a \sim_{\tr} b$ and $a \precsim_{\tr} b$ if $d_\tau(a) = d_\tau(b)$, respectively, $d_\tau(a) \le d_\tau(b)$, for all $\tau \in T(A)$.
\end{defi}

\begin{rema} \label{remark:comparison}
    Observe that
    $$a \precsim_s b \implies a \precsim b \implies a \precsim_{\tr} b, \qquad
        a \sim_s b \implies a \approx b \implies a \sim_{\tr} b$$
    for all positive elements $a$ and $b$ in any \Cs{} $A$.
    In Section \ref{sect:summary} we discuss under which conditions these implications can be reversed.
\end{rema}

%########################################################################################################################
%########################################################################################################################
\section{Open projections}

\noindent
    The bidual $A^{**}$ of a \Cs{} $A$ is equal to the von Neumann algebra arising as the weak closure of the image of $A$ under the universal representation $\pi_u \colon A \to \B(H_u)$ of $A$.
    Following Akemann, \cite[Definition II.1]{Ake69}, and Pedersen, \cite[Proposition 3.11.9, p.77]{Ped}), a projection $p$ in $A^{**}$ is said to be \emph{open} if it is the strong limit of an increasing sequence of positive elements from $A$, or, equivalently, if it belongs to the strong closure of the hereditary sub-\Cs{} $pA^{**}p \cap A$ of $A$.
    We shall denote this hereditary sub-\Cs{} of $A$ by $A_p$. (This agrees with the previous definition of $A_p$ if $p$ is a projection in $A$.)
    The map $p \mapsto A_p$  furnishes a bijective correspondence between open projections in $A$ and hereditary sub-\Cs s of $A$.
    The open projection corresponding to a hereditary sub-\Cs{} $B$ of $A$ is the projection onto the closure of the subspace $\pi_u(B)H_u$ of $H_u$. Let $\Proj_{\mathrm{o}}(A^{**})$ denote the set of open projections in $A^{**}$.

  A projection in $A^{**}$ is \emph{closed} if its complement is open.

    For each positive element $a$ in $A$ we let $p_a$ denote the open projection in $A^{**}$ corresponding to the hereditary sub-\Cs{} $A_a$ of $A$.
 Equivalently, $p_a$ is equal to the range projection of $\pi_u(a)$, and if $a$ is a contraction, then $p_a$ is equal to the strong limit of the increasing sequence $\{a^{1/n}\}$.
    Notice that $p_a = p_b$ if and only if $A_a = A_b$ if and only if $a \cong b$.
    If $A$ is separable, then each hereditary sub-\Cs{} of $A$ contains a strictly positive element and hence is of the form $A_a$ for some $a$.
    It follows that every open projection in $A^{**}$ is of the form $p_a$ for some positive element $a$ in $A$, whence there is a bijective correspondence between open projections in $A^{**}$ and positive elements in $A$ modulo the equivalence relation $\cong$.

\begin{pargr}[Closure of a projection]
    If $K\subseteq \Proj_{\mathrm{o}}(A^{**})$ is a family of open projections, then their supremum $\bigvee K$ is again open.
    Dually, the infimum of a family of closed projections is again closed.
    Therefore, if we are given any projection $p$, then we can define its \emph{closure} $\overline{p}$ as:
    \begin{align*}
        \overline{p}    &:=\bigwedge\{q\in \Proj(A^{**})\ :\ q \text{ is closed, } \, p\leq q\}.
    \end{align*}
\end{pargr}

\vspace{.3cm}
\noindent
We shall pursue various notions of comparisons and equivalences of open projections in $A^{**}$ that, via the correspondence $a \mapsto p_a$, match the notions of comparison and equivalences of positive elements in a \Cs{} considered in the previous section. First of all we have Murray-von Neumann equivalence $\sim$ and subequivalence $\precsim$ of projections in any von Neumann algebra. We shall show in Section~\ref{sec5}
that they correspond to tracial comparison. Peligrad and Zsid\'{o} made the following definition:

\begin{defi}[PZ-equivalence, {\cite[Definition 1.1]{PelZsi}}]
\label{defn:PZ-equiv}
    Let $A$ be a \Cs, and let $p$ and $q$ be open projections in $A^{**}$.
    Then $p,q$ are equivalent in the sense of Peligrad and Zsid\'{o} (PZ-equivalent, for short), denoted by $p\sim_{\PZ}q$, if there exists a partial isometry $v\in A^{**}$ such that
$$p=v^*v, \qquad q = vv^*, \qquad vA_p\subseteq A, \qquad v^*A_q\subseteq A.$$
    Say that $p\precsim_{\PZ}q$ if there exists $p'\in\Proj_{\mathrm{o}}(A^{**})$ such that $p\sim_{\PZ} p'\leq q$.
\end{defi}

\noindent
    PZ-equivalence is stronger than Murray-von Neumann equivalence.
    We will see in Section \ref{sect:summary} that it is in general strictly stronger, but the two equivalences do agree for some \Cs s and for some classes of projections.

We will now turn to the question of PZ-equivalence of left and right support projections. Peligrad and Zsid\'{o} proved in  \cite[Theorem 1.4]{PelZsi} that $p_{xx^*} \sim_{\PZ} p_{x^*x}$ for every $x \in A$ (and even for every $x$ in the multiplier algebra of $A$).  One can ask whether the converse is true. The following result gives a satisfactory answer.

\begin{prop}
\label{prop:PZ_implemented_by_left_right}
    Let $p,q\in\Proj_{\mathrm{o}}(A^{**})$ be two open projections with $p\sim_{\PZ}q$.
    If $p$ is the support projection of some element in $A$, then so is $q$,
    and in this case $p=p_{xx^*}$ and $q=p_{x^*x}$ for some $x\in A$.
\end{prop}

\begin{proof}
 There is a partial isometry $v$ in $A^{\ast\ast}$ with $p=v^\ast v$,  $vv^\ast=q$, and $vA_p\subseteq A$. This implies that $vA_pv^\ast\subseteq A$, so the map
 $x\mapsto vxv^\ast$ defines a $^*$-isomorphism from $A_p$ onto $A_q$. By assumption, $p=p_a$ for some positive element $a$ in $A$. Upon replacing $a$ by $\|a\|^{-1}a$ we can assume that $a$ is a contraction. Put
 $b:=vav^\ast\in A^+$. Then
    $$ p_b =\sup_n\ (vav^\ast)^{1/n} =\sup_n\ va^{1/n}v^\ast =v p_a v^\ast = q.$$
    Hence $q$ is a support projection, and moreover for $x:=va^{1/2}\in A$ we have $a=x^\ast x$ and  $xx^\ast=b$.
\end{proof}

\begin{rema}
    As noted above, every open projection in the bidual of a \emph{separable} \Cs{} is realized as a support projection, so that PZ-equivalence of two open projections means precisely that they are the left and right support projections of some elements in $A$.
\end{rema}

\begin{pargr}[Compact and closed projections]
\label{pargr:cpct_cld_proj}
    We define below an equivalence relation and an order relation on open projections that we shall show to match Cuntz comparison of positive elements (under the correspondence $a \mapsto p_a$).
    To this end we need to define the concept of compact containment, which is inspired by the notion of compact (and closed) projection developed by Akemann.

    The idea first appeared in \cite{Ake69}, although it was not given a name there, and it was later termed in the slightly different context of the atomic enveloping von Neumann algebra in \cite[Definition II.1]{Ake71}.
    Later again, it was studied by Akemann, Anderson, and Pedersen in the context of the universal enveloping von Neumann algebra (see \cite[after Lemma 2.4]{AkeAndPed}).

    A closed projection $p \in A^{**}$ is called \emph{compact} if there exists $a\in A^+$ of norm one such that $pa=p$.
    See \cite[Lemma 2.4]{AkeAndPed} for equivalent conditions.
    Note that a compact, closed projection $p \in A^{**}$ must be dominated by some positive element of $A$ (since $pa=p$ implies $p=apa\leq a^2\in A$).
    The converse also holds (this follows from the result \cite[Theorem II.5]{Ake71} transferred to the context of the universal enveloping von Neumann algebra).
\end{pargr}
%    Every closed projection in the bidual of a unital \Cs{} is compact.

\begin{defi}[Compact containment]
\label{defn:cpct_containment_opnProj}
    Let $A$ be a \Cs, and let $p,q\in\Proj_{\mathrm{o}}(A^{**})$ be open projections.
    We say that $p$ is \emph{compactly contained} in $q$ (denoted $p\ssubset q$) if $\overline{p}$ is a compact projection in $A_q$, i.e., if there exists a positive element $a$ in $A_q$ with $\|a\|=1$ and $\overline{p}a=\overline{p}$.

    Further, let us say that an open projection $p$ is \emph{compact} if it is compactly contained in itself, i.e., if $p\ssubset p$.
\end{defi}

\begin{prop}
\label{prop:cpct_opn_proj_lies_in_A}
    An open projection in $A^{**}$ is compact if and only if it belongs to $A$.
\end{prop}
\begin{proof}
    Every projection in $A$ is clearly compact.

    If $p$ is open and compact, then by definition there exists $a\in (A_p)^+$ such that $\overline{p}a=\overline{p}$.
    This implies that $p \leq\overline{p} \leq a \leq p$, whence $p = a \in A$.
\end{proof}

\begin{rema}
    Note that compactness was originally defined only for \emph{closed} projections in $A^{**}$ (see \ref{pargr:cpct_cld_proj}).
    In Definition \ref{defn:cpct_containment_opnProj} above we also defined a notion of compactness for \emph{open} projections in $A^{**}$ by assuming it to be compactly contained in itself.
    This should cause no confusion since, by Proposition \ref{prop:cpct_opn_proj_lies_in_A}, a \emph{compact, open} projection is automatically closed as well as compact in the sense defined for closed projections in \ref{pargr:cpct_cld_proj}.
\end{rema}

\noindent
    Now we can give a definition of \mbox{(sub-)}equivalence for open projections that we term Cuntz \mbox{(sub-)}equivalence, and which in the next section will be shown to agree with Cuntz \mbox{(sub-)}equivalence for positive elements and Hilbert modules in a \Cs. We warn the reader that our definition of Cuntz equivalence (below) does not agree with the notion carrying the same name defined by Lin in \cite{Lin10}. The latter was the one already studied by Peligrad and Zsid\'{o} that we (in Definition
\ref{defn:PZ-equiv}) have chosen to call Peligrad-Zsid\'{o} equivalence (or PZ-equivalence). Our definition below of Cuntz equivalence for open projections turns out to match the notion of Cuntz equivalence for positive elements, also when the \Cs{} does not have stable rank one.

\begin{defi}[Cuntz comparison of open projections]
\label{defn:Cu-comp_opnProj}
    Let $A$ be a \Cs, and let $p$ and $q$ be open projections in $A^{**}$.
    We say that $p$ is \emph{Cuntz subequivalent} to $q$, written $p\precsim_{\Cu} q$, if for every open projection $p'\ssubset p$ there exists an open projection $q'$ with $p'\sim_{\PZ} q'\ssubset q$.
    If $p\precsim_{\Cu} q$ and $q\precsim_{\Cu} p$ hold, then we say that $p$ and $q$ are \emph{Cuntz equivalent}, which we write as $p\sim_{\Cu} q$.
\end{defi}

%########################################################################################################################
%########################################################################################################################
\section{Comparison of positive elements and the corresponding relation on open projections}\label{section_comp_posi}

\noindent
We show in this section that the Cuntz comparison relation on positive elements corresponds to the Cuntz relation on the corresponding open projections. We also show that the Blackadar relation on positive elements, the Peligrad-Zsid\'{o} relation on their corresponding open projections, and isometric isomorphism of the corresponding Hilbert modules are equivalent.

\begin{pargr}[Hilbert modules]
    See \cite{AraPerToms} for a good introduction to Hilbert $A$-modules. Throughout this note all Hilbert modules are assumed to be right modules and countably generated.
 Let $A$ be a general \Cs.  We will denote by $\mathcal{H}(A)$ the set of all Hilbert $A$-modules.
    Every closed, right ideal in $A$ is in a natural way a Hilbert $A$-module. In particular, $E_a := \overline{aA}$ is a Hilbert $A$-module for every element $a$ in $A$.
    The assignment $a \mapsto E_a$ defines a natural map from the set of positive elements of $A$ to $\mathcal{H}(A)$.

If $E$ and $F$ are Hilbert $A$-modules, then $E$ is said to be \emph{compactly contained} in $F$, written $E \ssubset F$, if there exists a positive element $x$ in $\K(F)$, the compact operators of $\mathcal{L}(F)$, such
 that $xe = e$ for all $e \in E$.

For two Hilbert $A$-modules $E,F$ we say that $E\precsim_{\Cu} F$
($E$ is \emph{Cuntz subequivalent} to $F$) if for every Hilbert $A$-submodule $E' \ssubset E$ there exists $F'\ssubset F$ with $E'\cong F'$ (isometric isomorphism). Further declare $E \approx F$ (\emph{Cuntz equivalence}) if $E\precsim_{\Cu} F$ and $F\precsim_{\Cu} E$.
\end{pargr}

\noindent
Before relating the Blackadar relation with the Peligrad-Zsid\'{o} relation we prove the following lemma restating the Blackadar relation:

\begin{lem}
\label{lem:TFAE:equivalence}
    Let $A$ be a \Cs, and let $a$ and $b$ be positive elements in $A$. The following conditions are equivalent:
    \begin{enumerate}
        \item $a\sim_s b$,
        \item
        there exist $a',b'\in A^+$ with $a\cong a'\sim b'\cong b$,
       \item there exists $x \in A$ such that $A_a = A_{x^*x}$ and $A_b = A_{xx^*}$,
        \item
        there exists $b'\in A^+$ with $a\sim b'\cong b$,
        \item
        there exists $a'\in A^+$ with $a\cong a'\sim b$.
        \end{enumerate}
\end{lem}

\begin{proof}
    (ii) is just a reformulation of (i), and (iii) is a reformulation of (ii) keeping in mind that $A_c = A_d$ if and only if $c \cong d$.

(iv) $\Rightarrow$ (ii) and (v) $\Rightarrow$ (ii) are trivial.

(iii) $\Rightarrow$ (v): Take $x \in A$ such that $A_a = A_{x^*x}$ and $A_b = A_{xx^*}$. Let $x = v|x|$ be the polar decomposition for $x$ (with $v$ a partial isometry in $A^{**}$). Then $c \mapsto v^*cv$ defines an isomorphism from  $A_{xx^*} = A_b$ onto $A_{x^*x} = A_a$. This isomorphism maps the strictly positive element $b$ of $A_b$ onto a strictly positive element $a' = v^*bv$ of $A_a$ . Hence $b \sim a' \cong a$ as desired.

The proof of (iii) $\Rightarrow$ (iv) is similar.
\end{proof}

\noindent
The equivalence of (i) and (iv) in the proposition below was noted to hold in Lin's recent paper, \cite{Lin10}. We include a short proof of this equivalence for completeness.

\begin{prop}
\label{prop:TFAE:equivalence}
    Let $A$ be a \Cs, and let $a$ and $b$ be positive elements in $A$. The following conditions are equivalent:
    \begin{enumerate}
        \item $a\sim_s b$,
        \item $E_a$ and $E_b$ are isomorphic as Hilbert $A$-modules,
        \item there exists $x \in A$ such that $E_a = E_{x^*x}$ and $E_b = E_{xx^*}$,
        \item $p_a \sim_{\PZ} p_b$.
        \end{enumerate}
\end{prop}

\begin{proof}
    (i) $\Rightarrow$ (iv): As remarked earlier, it was shown in \cite[Theorem 1.4]{PelZsi}  that $p_{x^*x} \sim_{\PZ} p_{xx^*}$ for all $x \in A$. In other words, $a \sim b$ implies $p_a \sim_{\PZ}  p_b$. Recall also that $p_a = p_b$ when $a \cong b$. These facts prove the implication.

(iv) $\Rightarrow$ (i): If $p_a \sim_{\PZ} p_b$, then by Proposition \ref{prop:PZ_implemented_by_left_right}, there exists positive elements $a'$ and $b'$ in $A$ such that $p_a = p_{a'}$, $p_b = p_{b'}$, and $a' \sim b'$. Now, $p_a = p_{a'}$ and $p_b = p_{b'}$ imply that $a \cong a'$ and $b \cong b'$, whence (i) follows (see also Lemma \ref{lem:TFAE:equivalence} ).

 (ii) $\Rightarrow$ (iii):
    Let $\Phi \colon E_a\to E_b$ be an isomorphism of Hilbert $A$-modules, i.e., a bijective $A$-linear map preserving the inner product.
    Set $x:=\Phi(a)\in E_b$.
    Then
$$\overline{xA}=\overline{\Phi(a)A}=\overline{\Phi(aA)} = E_b,$$ whence $E_b=E_x=E_{xx^*}.$
    Since $\Phi$ preserves the inner product,
$$a^2 = \langle a,a\rangle_{E_a}=\langle\Phi(a),\Phi(a)\rangle_{E_b}=x^\ast x.$$
Hence $E_a = E_{a^2} = E_{x^*x}$ and $E_b = E_{xx^*}$.

    (iii) $\Rightarrow$ (ii):
    Let $x=v|x|$ be the polar decomposition of $x$ in $A^{\ast\ast}$.
    Note that $E_{|x|}=E_{x^\ast x}$ and $E_{xx^\ast}=E_{|x^*|}$.
    Define an isomorphism $E_{|x|}\to E_{|x^*|}$ by $z\mapsto vz$.

    (i) $\Leftrightarrow$ (iii):
    This follows from the one to one correspondence between hereditary sub-\Cs s and right ideals:
    A hereditary sub-\Cs{} $B$ corresponds to the right ideal $\overline{BA}$, and, conversely, a right ideal $R$ corresponds the hereditary algebra $R^*R$.
    In particular, $E_a=\overline{A_aA}$ and $A_a=E_a^*E_a$.

    If (i) holds, then, by Lemma \ref{lem:TFAE:equivalence},  $A_a=A_{x^*x}$ and $A_{xx^*}=A_b$ for some $x\in A$. This shows that
    $E_a=\overline{A_aA}=\overline{A_{x^*x}A}=E_{x^*x}$ and, similarly, $E_b=E_{xx^*}$.

    In the other direction, if $E_a=E_{x^*x}$ and $E_{xx^*}=E_b$ for some $x\in A$,
    then $A_a=E_a^*E_a=E_{x^*x}^*E_{x^*x}=A_{x^*x}$ and, similarly, $A_b=A_{xx^*}$, whence $a \sim_s b$.
\end{proof}

\begin{pargr} \label{comment}
It follows from the proof of (ii) $\Rightarrow$ (iii) of the proposition above that if $a$ is a positive element in a \Cs{} $A$ and if $F$ is a Hilbert $A$-module such that $E_a \cong F$, then $F = E_b$ for some positive element $b$ in $A$. In fact, if $\Phi \colon E_a \to F$ is an isometric isomorphism, then we can take $b$ to be $\Phi(a)$ as in the before mentioned proof.
\end{pargr}

\begin{pargr}
\label{prop:TFAE:inclusion_equality}
For any pair of positive elements $a$ and $b$ in a \Cs{} $A$ we have the following equivalences:
 $$ a\in A_b \iff  A_a\subseteq A_b \iff  E_a\subseteq E_b \iff p_a\leq p_b,$$
as well as the following equivalences:
  $$ a\in A_b \, \, \text{and} \, \, b\in A_a \iff  a\cong b \iff A_a=A_b \iff E_a=E_b \iff p_a=p_b.$$
\end{pargr}

\noindent
As a consequence of Proposition \ref{prop:TFAE:equivalence}, Lemma \ref{lem:TFAE:equivalence}, and the remark above we obtain the following proposition:

\begin{prop} \label{prop:TFAE:subequivalence}
    Let $A$ be a \Cs, and let $a$ and $b$ be positive elements in $A$. The following conditions are equivalent:
    \begin{enumerate}
        \item
        $a\precsim_s b$,
        \item
        there exists a Hilbert $A$-module $E'$ such that $E_a\cong E'\subseteq
        E_b$,
        \item
        there exists $x\in A$ with $E_a=E_{x^*x}$ and $E_{xx^*}\subseteq
        E_b$,
        \item
        $p_a\precsim_{\PZ} p_b$.
    \end{enumerate}
\end{prop}

\begin{lem} \label{lm:unit}
Let $a$ and $e$ be positive elements in a \Cs{} $A$ and assume that $e$ is a contraction. Then the following equivalences hold:
$$ae = a \iff p_a e = p_a \iff \overline{p}_a e = \overline{p}_a.$$
\end{lem}

\begin{proof} The two "$\Leftarrow$"-implications are trivial. Suppose that $ae=a$. Let $\chi$ be indicator function for the singleton $\{1\}$, and put $q = \chi(e) \in A^{**}$. Then $qe = q$ and $q$ is the largest projection in $A^{**}$ with this property. Moreover, $q$ is the projection onto the kernel of $1-e$, hence $1-q$ is the projection onto the range of $1-e$, i.e., $1-q = p_{1-e}$. This shows that $q$ is a closed projection. As $a$ and $1-e$ are orthogonal so are their range projections $p_a$ and $p_{1-e}$, whence $p_a \le 1-p_{1-e} = q$. Thus $\overline{p}_a \le q$. This shows that $\overline{p}_a e = \overline{p}_a$.
\end{proof}

\begin{lem}
\label{prop:lma1}
    Let $A$ be a \Cs, and let $e$ and $a$ be positive elements in $A$.
    If $ae=a$, then $\overline{p}_a\leq p_e$.
\end{lem}
\begin{proof}
 Upon replacing $e$ with $f(e)$, where $f \colon \R^+ \to \R^+$ is given by $f(t) = \max\{t,1\}$, we may assume that $e$ is a contraction. If $ae=a$, then $\overline{p}_a e = \overline{p}_a$ by Lemma \ref{lm:unit}, and this implies that $\overline{p}_a\leq p_e$.
\end{proof}

\noindent We show below that the two previously defined notions of compact containment agree. To do so we introduce a third notion of compact containment:

\begin{defi} \label{def:a<<b}
    Let $a$ and $b$ be positive elements in a \Cs.
    Then $a$ is said to be compactly contained in $b$, written $a \ssubset b$, if and only if there exists a positive element $e$ in $A_b$ such that $ea = a$.
\end{defi}

\noindent Following the proof of Lemma \ref{prop:lma1} , the element $e$ above can be assumed to be a contraction.

\begin{prop}
\label{prop:TFAE:cpct_containment}
    Let $A$ be a \Cs, let $b$ be a positive element in $A$, and let $a$ be a positive element in $A_b$.  Then the following statements are equivalent:
    \begin{enumerate}
        \item
        $E_a\ssubset E_b$,
        \item
        $a \ssubset b$,
       \item
        $p_a\ssubset p_b$,
        \item
        $\overline{p}_a \leq p_b$ and $\overline{p}_a$ is compact in $A$.
    \end{enumerate}
\end{prop}

\begin{proof}
    (i) $\Leftrightarrow$ (ii): By definition, (i) holds if and only if there exists a positive element $e$ in $\K(E_b)$, such that $e$ acts as the identity on $E_a$. We can identify $\K(E_b)$ with $A_b$, as elements of the latter act on $E_b$ by left-multiplication. Thus (i) is equivalent to the existence of a positive element $e$ in $A_b$ such that $ex=x$ for all $x\in E_a =\overline{aA}$.
    The latter condition is fulfilled precisely if $ea=a$.

    (ii) $\Leftrightarrow$ (iii): (iii) holds if and only if there exists a positive element $e$ in $A_b$ such that $\overline{p}_a e = \overline{p}_a$; and (ii) holds if and only if there exists a positive element $e$ in $A_b$ such that $ae=a$. In both cases $e$ can be taken to be a contraction, cf.\ the proof of Lemma \ref{prop:lma1}. The bi-implication now follows from Lemma \ref{lm:unit}.

    (ii) and (iii) $\Rightarrow$ (iv): If $a \ssubset b$, then there is a positive contraction $e$ in $A_b$ such that $ae=a$. By Lemma \ref{prop:lma1} this implies that $\overline{p}_a\le p_e \le p_b$. From (iii) we have that $\overline{p}_a$ is compact in $A_b$ which entails that $\overline{p}_a$ also is compact in $A$.

    (iv) $\Rightarrow$ (iii): This is \cite[Lemma 2.5]{AkeAndPed}.
\end{proof}

\begin{rema}
    In many cases it is automatic that $\overline{p}$ is compact, and then $p\ssubset q$ is equivalent to the condition $\overline{p}\leq q$.
    For example, if $A$ is unital, then all closed projections in $A^ {**}$ are compact.
    More generally if $a\in A^+$ sits in some corner of $A$, then $\overline{p}_a$ is compact.
\end{rema}

\begin{lem} \label{lem:a}
Let $a$ be a positive element in a \Cs{} $A$.
\begin{enumerate}
\item If $E'$ is a Hilbert $A$-module that is compactly contained in $E_a$, then $E' \subseteq E_{(e-\ep)_+}$ for some positive element $e \in A_a$ and some $\ep > 0$.
\item If $q,q'$ are open projections in $A^{**}$ such that $q'$ is compactly contained in $q$, then $q' \le p_{(e-\ep)_+}$ for some positive element $e \in A_q$ and some $\ep > 0$
\end{enumerate}
\end{lem}

\begin{proof}
(i): By definition there is a positive element $e$ in $\K(E_a) = A_a$ such that $ex = x$ for all $x \in E'$. This implies that $(e-1/2)_+x = \frac{1}{2}x$ for all $x \in E'$, whence $E' \subseteq E_{(e-1/2)_+}$.

(ii): If $q'$ is compactly contained in $q$, then there is a positive element $e$ in $A_{q}$ such that $q'e = q'$ (in fact such that $\overline{q'} e = \overline{q'}$). It follows that $q' (e-1/2)_+ = \frac{1}{2} q'$, and hence that $q' \le  p_{(e-1/2)_+}$.
\end{proof}

\begin{prop}
\label{prop:TFAE:Cuntz_subequivalence}
    Let $a$ and $b$ be positive elements in a \Cs{} $A$.
    Then the following statements are equivalent:
    \begin{enumerate}
        \item
        $a\precsim b$.
        \item
        $E_a\precsim_{\Cu} E_b$.
        \item
        $p_a \precsim_{\Cu} p_b$.
    \end{enumerate}
\end{prop}
\begin{proof}
 The equivalence of (i) and (ii) was first shown in \cite[Appendix]{CowEllIva}, see also
\cite[Theorem 4.33]{AraPerToms}.

(ii) $\Rightarrow$ (iii): Suppose that $E_a\precsim_{\Cu} E_b$, and let $p'$ be an arbitrary open projection in $A^{**}$ which is compactly contained in $p_a$. Then, by Lemma \ref{lem:a}, $p' \le p_{(e-\ep)_+}$ for some positive element $e$ in $A_a$ and some $\ep > 0$. Notice that $(e-\ep)_+ \ssubset a$.
It follows from Proposition \ref{prop:TFAE:cpct_containment} that $E_{(e-\ep)_+}$ is compactly contained in $E_a$. Accordingly, $E_{(e-\ep)_+} \cong F'$ for some Hilbert $A$-module $F'$ that is compactly contained in $E_b$. By \ref{comment}, $F' = E_c$ for some positive element $c$ in $A$. It now follows from Proposition \ref{prop:TFAE:cpct_containment} and from Proposition \ref{prop:TFAE:equivalence} that
$$p' \le p_{(e-\ep)_+}  \sim_{\PZ} p_c \ssubset p_b.$$
This shows that $p_a \precsim_{\Cu} p_b$.

(iii) $\Rightarrow$ (ii): Suppose that $p_a\precsim_{\Cu} p_b$, and let $E'$ be an arbitrary Hilbert $A$-module which is compactly contained in $E_a$. Then, by Lemma \ref{lem:a}, $E' \subseteq E_{(e-\ep)_+}$ for some positive element $e$ in $A_a$ and some $\ep > 0$.
It follows from Proposition \ref{prop:TFAE:cpct_containment} that $p_{(e-\ep)_+}$ is compactly contained in $p_a$. Accordingly, $p_{(e-\ep)_+} \sim_{\PZ} q'$ for some open projection $q'$ in $A^{**}$ that is compactly contained in $p_b$. By Proposition \ref{prop:PZ_implemented_by_left_right}, $q' = p_c$ for some positive element $c$ in $A$. It now follows from Proposition \ref{prop:TFAE:cpct_containment} and from Proposition \ref{prop:TFAE:equivalence} that
$$E' \subseteq E_{(e-\ep)_+}  \cong E_c \ssubset E_b.$$
This shows that $E_a \precsim_{\Cu} E_b$.
\end{proof}

\noindent By the definition of Cuntz equivalence of positive elements, Hilbert $A$-modules, and of open projections, the proposition above immediately implies the following:

\begin{corol}
\label{cor:TFAE:Cuntz_equivalence}
 For every pair of positive elements $a$ and $b$ in a \Cs{} $A$ we have the following equivalences:
 $$a\approx b \iff E_a\approx E_b \iff p_a \sim_{\Cu} p_b.$$
\end{corol}

\noindent
    We conclude this section by remarking that the pre-order $\precsim_{\PZ}$ on the open projections is not algebraic (unlike the situation for Murray-von Neumann subequivalence).
    Indeed, if $p$ and $q$ are open projections $A^{**}$ with $p\leq q$, then $q-p$ need not be an open projection.
    For the same reason, $\precsim_{\Cu}$ is not an algebraic order.
    However, Cuntz comparison is approximately algebraic in the following sense.

\begin{prop}
    Let $A$ be a \Cs{}, and let $p,p',q\in A^{**}$ be open projections with $p'\ssubset p\precsim_{\Cu} q$.
    Then there exists an open projection $r\in A^{**}$ such that $p'\oplus r\precsim_{\Cu} q\precsim_{\Cu} p\oplus r$ .
\end{prop}

\begin{proof}
    By Lemma \ref{lem:a} (ii) there exists an open projection $p''$ with $p'\ssubset p''\ssubset p$ (take $p''$ to be $p_{(a-\ep/2)_+}$ in that lemma).
    By the definition of Cuntz sub-equivalence there exists an open projection $q''$ such that $p''\sim_{\PZ} q'' \ssubset q$.
    Since $p''\sim_{\PZ} q''$ implies $p''\sim_{\Cu} q''$, there exists an open projection $q'$  with $p'\sim_{\PZ} q'\ssubset q''$. Then $r:=q-\overline{q'}$ is an open projection.

    Since $q'\ssubset q''$ implies $\overline{q'}\leq q''$, and $q' \le \overline{q'}$, we get
$$p' \oplus r \sim_{\PZ} q' \oplus r \precsim_{\PZ} q = \overline{q'}+r \precsim_{\Cu}  q''\oplus r \sim_{\PZ}  p''\oplus r \precsim  p\oplus r$$
as desired.
\end{proof}

\noindent
    Translated, this result says that for positive elements $a',a,b$ in $A$ with $a'\ssubset a\precsim b$ there exists a positive element $c$ such that $a'\oplus c\precsim b\precsim a\oplus c$.

    To formulate the result in the ordered Cuntz semigroup, we recall that an element $\alpha\in\Cu(A)$ is called \emph{way-below} $\beta\in\Cu(A)$, denoted $\alpha\ll\beta$, if for every increasing sequence $\{\beta_k\}$ in $\Cu(A)$ with $\beta\leq\sup_k\beta_k$ there exists $l\in\N$ such that already $\alpha\leq\beta_{l}$.
    Consequently, in the Cuntz semigroup we get the following almost algebraic order:

\begin{corol}[Almost algebraic order in the Cuntz semigroup]
    Let $A$ be a \Cs, and let $\alpha', \alpha, \beta$ in $\Cu(A)$ be such that $\alpha'\ll\alpha\le\beta$.
    Then there exists $\gamma\in \Cu(A)$ such that $\alpha'+\gamma \le \beta \le \alpha+\gamma$.
\end{corol}

%########################################################################################################################
%########################################################################################################################
\section{Comparison of projections by traces.}\label{sec5}

\noindent
    In this section we show that Murray-von Neumann \mbox{(sub-)}equivalence of open projections in the bidual of a separable \Cs{} is equivalent to tracial comparison of the corresponding positive elements of the \Cs.  For the proof we need to show that every lower semicontinuous tracial weight on a \Cs{} extends (not necessarily uniquely) to a normal tracial weight on its bidual; and that Murray-von Neumann comparison of projections in any von Neumann algebra "that is not too big" is determined by tracial weights. We expect those two results to be known to experts, but in lack of a reference and for completeness we have included their proofs.

Recall that a weight $\varphi$ on a \Cs{} $A$ is an additive map $\varphi \colon A^+
\to [0,\infty]$  satisfying $\varphi(\lambda a) = \lambda \varphi(a)$ for all $a \in A^+$ and all $\lambda
 \in \R^+$. We say that $\varphi$ is \emph{densely defined} if the set $\{a \in A^+ : \varphi(a) < \infty\}$ is dense in $A^+$. Recall from Section \ref{sec2} that the set of (norm) lower semicontinuous tracial weights on a $A$ in this paper is denoted by $T(A)$.

If $M$ is a von Neumann algebra, then let $W(M)$ denote the set of  \emph{normal} weights on $M$, and let $W_{\tr}(M)$ denote the set of normal tracial weights on $M$, i.e., weights $\varphi$ for which $\varphi(x^*x) = \varphi(xx^*)$ for all $x \in M$. The standard trace on $\B(H)$ is an example of a normal tracial weight.

For the extension of weights on a \Cs{} to its universal enveloping von Neumann algebra, we use the result below from \cite[Proposition 4.1 and Proposition 4.4]{Comb}. For every $f$ in the dual $A^*$ of a \Cs{} $A$, let $\widetilde{f}$ denote the unique normal extension of $f$ to $A^{**}$. (One can equivalently obtain $\widetilde{f}$ via the natural pairing: $\widetilde{f}(z) = \langle f, z \rangle$ for $z \in A^{**}$.)

\begin{prop}[{Combes, \cite{Comb}}]
\label{prop:extension_weight}
    Let $A$ be a \Cs{}, let $\varphi \colon A^+ \to [0,\infty]$ be a densely defined lower semicontinuous       weight. Define a map $\widetilde{\varphi} \colon (A^{**})^+ \to [0,\infty]$ by:
    $$\widetilde{\varphi}(z):=\sup\{\widetilde{f}(z)\ :\ f\in A^*, \,0 \le  f\leq\varphi\} , \qquad z\in (A^{**})^+.$$
Then $\widetilde{\varphi}$ is a normal weight on $A^{**}$ extending $\varphi$.
    Moreover, if $\varphi$ is tracial, then $\widetilde{\varphi}$ is the unique extension of $\varphi$ to a normal weight on $A^{**}$.
\end{prop}

\noindent Combes did not address the question whether the (unique) normal weight on $A^{**}$ that extends a densely defined lower semicontinuous tracial weight on $A$ is itself a trace. An affirmative answer to this question is included in the proposition below. 

\begin{prop}
\label{prop:extension_tracial_weight}
    Let $A$ be a \Cs{}, and let $\varphi$ be a lower semicontinuous tracial weight on $A$.
    Then there exists a normal, tracial weight on $A^{**}$ that extends $\varphi$.
\end{prop}

\begin{proof} The closure of the linear span of the set $\{a\in A^+\ :\ \varphi(a)<\infty\}$ is a closed two-sided ideal in $A$. Denote it by $I_\varphi$. The restriction of $\varphi$ to $I_\varphi$ is a densely defined tracial weight, which therefore, by Combes' extension result (Proposition \ref{prop:extension_weight}), extends (uniquely) to a normal weight $\widehat{\varphi}$ on $I_{\varphi}^{**}$. The ideal $I_\varphi$ corresponds to an open central projection $p$ in $A^{**}$ via the identification $I_\varphi = A^{**}p \cap A$, and $I_\varphi^{**} = A^{**} p$. In other words, $I_\varphi^{**}$ is a central summand in $A^{**}$. Extend $\varphi$ to a normal weight $\widetilde{\varphi}$ on the positive elements in $A^{**}$ by the formula
    \begin{align*}
        \widetilde{\varphi}(z)
        &=\begin{cases}
            \widehat{\varphi}(z), &\text{ if $z\in I_\varphi^{**}$}, \\
            \infty,  &\text{ otherwise.} \\
        \end{cases}
    \end{align*}
It is easily checked that $\widetilde{\varphi}$ is a normal weight that extends $\varphi$, and that $\widetilde{\varphi}$ is tracial if we knew that $\widehat{\varphi}$ is tracial. To show the latter, upon replacing $A$ with $I_\varphi$, we can assume that $\varphi$ is densely defined.

We proceed to show that $\widetilde{\varphi}$ is tracial under the assumption that $\varphi$ is densely defined. To this end it suffices to show that $\widetilde{\varphi}$ is unitarily invariant, i.e.,  that $\widetilde{\varphi}(uzu^*)=\widetilde{\varphi}(z)$ for all unitaries $u$ in $A^{**}$ and all positive elements $z$ in $A^{**}$.
We first check this when the unitary $u$ lies in $\widetilde{A}$, the unitization of $A$, which we view as a unital sub-\Cs{} of $A^{**}$, and for an arbitrary positive element $z$ in $A^{**}$. For each $f$ in $A^*$ let $u.f$ denote the functional in $A^*$ given by $(u.f)(a) = f(uau^*)$ for $a \in A$.
By the trace property of $\varphi$ we see that if $f \in A^*$ is such that $0 \le f \le \varphi$, then also $0 \le u.f \le \varphi$, and vice versa since $f = u^*.(u.f)$.
It follows that
    \begin{align*}
        \widetilde{\varphi}(uzu^*)
            &=\sup\{\widetilde{f}(uzu^*) : f\in A^*, \,  0 \le f\leq\varphi\}
         \, = \, \sup\{\widetilde{u.f} (z) : f \in A^*, \, 0 \le f\leq\varphi\}  \\
            &=\sup\{\widetilde{f}(z) : f\in A^*,  \, 0 \le f\leq\varphi\}  \,  = \, \widetilde{\varphi}(z).
    \end{align*}

    For the general case we use Kaplansky's density theorem (see \cite[Theorem 2.3.3, p.25]{Ped}), which says that the unitary group $U(\widetilde{A})$ is $\sigma$-strongly dense in $U(A^{**})$.
    Thus, given $u$ in $U(A^{**})$ we can find a net $(u_\lambda)$ in $U(\widetilde{A})$
converging $\sigma$-strongly to $u$. It follows that $(u_\lambda z u_\lambda^*)$ converges $\sigma$-strongly (and hence $\sigma$-weakly) to $uzu^*$. As $\widetilde{\varphi}$ is $\sigma$-weakly lower semicontinuous (see \cite[III.2.2.18, p.\ 253]{Bla}), we get
$$\widetilde{\varphi}(uzu^*) =\widetilde{\varphi}(\lim_\lambda u_\lambda zu_\lambda^*)
            \leq\lim_\lambda\widetilde{\varphi}(u_\lambda zu_\lambda^*) =\widetilde{\varphi}(z).$$
The same argument shows that $\widetilde{\varphi}(z) =\widetilde{\varphi}(u^*(uzu^*)u) \leq\widetilde{\varphi}(uzu^*)$. This proves that $\widetilde{\varphi}(uzu^*)=\widetilde{\varphi}(z)$ as desired.
\end{proof}

\noindent The extension $\widetilde{\varphi}$ in Proposition \ref{prop:extension_tracial_weight} need not be unique if $\varphi$ is not densely defined. Take for example the trivial trace $\varphi$ on the Cuntz algebra $\cO_2$ (that is zero on zero and infinite elsewhere). Then every normal tracial weight on $\cO_2^{**}$ that takes non-zero (and hence the value $\infty$) on every properly infinite projection is an extension of $\varphi$, and there are many such normal tracial weights arising from the type I$_\infty$ and type II$_\infty$ representations $\cO_2$.  On the other hand, every densely defined lower semicontinuous tracial weight on a \Cs{} extends uniquely to a normal tracial weight on its bidual by Combes' result (Proposition \ref{prop:extension_weight}) and by Proposition \ref{prop:extension_tracial_weight}.

\begin{rema} \label{rem:traces}
Given a \Cs{} $A$ equipped with a lower semicontinuous tracial weight $\tau$ and a positive element $a$ in $A$. Then we can associate to $\tau$ a dimension function $d_\tau$ on $A$ (as above Definition
\ref{def:trace-comparison}). Let $\widetilde{\tau}$ be (any) extension of $\tau$ to a normal tracial weight on $A^{**}$ (cf.\  Proposition \ref{prop:extension_tracial_weight}). Then $d_\tau(a) = \widetilde{\tau}(p_a)$. To see this, assume without loss of generality that $a$ is a contraction. Then $p_a$ is the strong operator limit of the increasing sequence $\{a^{1/n}\}$, whence
$$d_\tau(a) = \lim_{n \to \infty} \tau(a^{1/n}) = \lim_{n \to \infty} \widetilde{\tau}(a^{1/n}) = \widetilde{\tau}(p_a)$$
by normality of $\widetilde{\tau}$.
\end{rema}

\begin{corol}
\label{prop:MvN-comp_implies_LDF-comp}
    Let $a$ and $b$ be positive elements in a \Cs{} $A$.
    If $p_a\precsim p_b$ in $A^{**}$, then $a \precsim_{\tr} b$ in $A$; and if $p_a \sim p_b$ in $A^{**}$, then $a \sim_{\tr} b$ in $A$.
\end{corol}

\begin{proof}
    Suppose that $p_a\precsim p_b$ in $A^{**}$. Then $\omega(p_a)\leq\omega(p_b)$ for every tracial weight $\omega$ on $A^{**}$.

    Now let $\tau\in T(A)$ be any lower semicontinuous tracial weight, and let $d_\tau$ be the corresponding dimension function.
    By Proposition \ref{prop:extension_tracial_weight}, $\tau$ extends to a tracial, normal weight $\widetilde{\tau}$ on $A^{**}$. Using the remark above, it follows that $d_\tau(a) = \widetilde{\tau}(p_a) \le \widetilde{\tau}(p_b) = d_\tau(b)$. This proves that $a \precsim_{\tr} b$. The second statement in the corollary follows from the first statement.
\end{proof}

\noindent
    We will now show that the converse of Corollary \ref{prop:MvN-comp_implies_LDF-comp} is true for separable \Cs s.
    First we need to recall some facts about the dimension theory of (projections in) von Neumann algebras.
    A good reference is the recent paper \cite{Sher} of David Sherman.

\begin{defi}[{Tomiyama \cite[Definition 1]{Tom}, see also \cite[Definition 2.3]{Sher}}]
\label{defn:homogeneousElt}
    Let $M$ be a von Neumann algebra, $p\in \Proj(M)$ a non-zero projection, and $\kappa$ a cardinal.
    Say that $p$ is $\kappa$-\emph{homogeneous} if $p$ is the sum of $\kappa$ mutually equivalent projections, each of which is the sum of centrally orthogonal $\sigma$-finite projections. Set
    \begin{align*}
        \kappa_M    &:=\sup\{\kappa\ :\ M\text{ contains a $\kappa$-homogeneous element} \}.
    \end{align*}
\end{defi}

\noindent
A projection can be $\kappa$-homogeneous for at most one $\kappa\geq\aleph_0$; and if $\kappa \ge \aleph_0$, then two $\kappa$-homogeneous projections are equivalent if they have identical central support (see \cite{Tom} and \cite{Sher}).  We shall use these facts in the proof of Proposition \ref{prop:comparison_proj_by_traces}.

    But first we show that the enveloping von Neumann algebra $A^{**}$ of a separable \Cs{} $A$ has $\kappa_{A^{**}}\leq\aleph_0$, a property that has various equivalent formulations and consequences (see \cite[Propositions 3.8 and 5.1]{Sher}).
    This property is useful, since it means that there are no issues about different ''infinities''.
   For instance, the set of projections up to Murray-von Neumann equivalence in an arbitrary  II$_\infty$ factor $M$ (not necessarily with separable predual) can be identified with $[0,\infty)\cup\{\kappa : \aleph_0\leq\kappa\leq\kappa_M\}$, see \cite[Corollary 2.8]{Sher}. Thus, tracial weights on $M$ need not separate projections up to equivalence. However, if $\kappa_M\leq\aleph_0$, then normal, tracial weights on $M$ do in fact separate projections up to Murray-von Neumann equivalence.

\begin{lem}
\label{prop:sep_C-alg_aleph_zero}
    Let $A$ be a separable \Cs.
    Then $\kappa_{A^{**}}\leq\aleph_0$.
\end{lem}

\begin{proof} We show the stronger statement that whenever $\{p_i\}_{i \in I}$ is a family of non-zero pairwise equivalent and orthogonal projections in $A^{**}$, then $\mathrm{card}(I) \le \aleph_0$.
    The universal representation $\pi_u$ of $A$ is given as $\pi_u  =\bigoplus_{\varphi\in S(A)}\pi_\varphi$,
where $S(A)$ denotes the set of states on $A$, and where $\pi_\varphi \colon  A\to \B(H_\varphi)$ denotes the GNS-representation corresponding to the state $\varphi$. It follows that
$$A^{**} = \pi_u(A)^{''} \subseteq \bigoplus_{\varphi\in S(A)} \B(H_\varphi).$$
The projections $\{p_i\}_{i \in I}$ are non-zero in at least one summand $\B(H_\varphi)$; but then $I$ must be countable because each $H_\varphi$ is separable.
\end{proof}

\begin{prop}
\label{prop:comparison_proj_by_traces}
    Let $M$ be a von Neumann algebra with $\kappa_{M}\leq\aleph_0$, and let $p,q\in \Proj(M)$ be two projections.
    Then $p\precsim q$ if and only if $\omega(p)\leq\omega(q)$ for all normal tracial weights $\omega$ on $M$.
\end{prop}

\begin{proof}
    The ''only if'' part is obvious. We prove the ''if" part and assume accordingly that $\omega(p) \le \omega(q)$ for all normal tracial weights $\omega$ on $M$, and we must show that $p \precsim q$.  We show first that it suffices to consider the case where $q \le p$.

There is a central projection $z$ in $M$ such that $zp \precsim zq$ and $(1-z)p \succsim (1-z)q$. We are done if we can show that $(1-z)p \precsim (1-z)q$.
Every normal tracial weight on $(1-z)M$ extends to a normal tracial weight on $M$ (for example by setting it equal to zero on $zM$), whence our assumptions imply that $\omega((1-z)p) \le \omega((1-z)q)$ for all tracial weights $\omega$ on $(1-z)M$. Upon replacing $M$ by $(1-z)M$, and $p$ and $q$ by $(1-z)p$ and $(1-z)q$, respectively, we can assume that $p \succsim q$, i.e., that $q \sim q' \le p$ for some projection $q'$ in $M$. Upon replacing $q$ by $q'$ we can further assume that $q \le p$ as desired.

There is a central projection $z$ in $M$ such that $zq$ is finite and $(1-z)q$ is properly infinite (see \cite[6.3.7, p.\ 414]{KadRing2}). Arguing as above it therefore suffices to consider the two cases where $q$ is finite and where $q$ is properly infinite.

Assume first that $q$ is finite. We show that $p=q$. Suppose, to reach a contradiction, that $p-q\ne 0$. Then there would be a normal tracial weight $\omega$ on $M$ such that $\omega(q) = 1$ and $\omega(p-q) > 0$. But that would entail that $\omega(p) > \omega(q)$ in contradiction with our assumptions. To see that $\omega$ exists, consider first the case where $q$ and $p-q$ are not centrally orthogonal, i.e., that $c_q c_{p-q} \ne 0$. Then there are non-zero projections $e \le q$ and $f \le p-q$ such that $e \sim f$. Choose a normal tracial state $\tau$ on the finite von Neumann algebra $qMq$ such that $\tau(e) > 0$. Then $\tau$ extends uniquely to a normal tracial weight $\omega_0$ on $Mc_q$ and further to a normal tracial weight $\omega$ on $M$ by the recipe $\omega(x) = \omega_0(xc_q)$. Then $\omega(q) = \tau(q) = 1$ and $\omega(p-q) \ge \omega_0(f) = \omega_0(e) = \tau(e) >0$. In the case where $q$ and $p-q$ are centrally orthogonal, take a normal tracial weight $\omega_0$ (for example as above) such that $\omega_0(q) = 1$ and extend $\omega_0$ to a normal tracial weight $\omega$ on $M$ by the recipe $\omega(x) = \omega_0(x)$ for all positive elements $x \in Mc_q$ and $\omega(x) = \infty$ whenever $x$ is a positive element in $M$ that does not belong to $Mc_q$. Then $\omega(q) = 1$ and $\omega(p-q) = \infty$.

Assume next that $q$ is properly infinite.
Every properly infinite projection can uniquely be written as a central sum of homogeneous projections (see \cite[Theorem 1]{Tom}, see also \cite[Theorem 2.5]{Sher} and the references cited there).
By the assumption that $\kappa_{M}\leq\aleph_0$ we get that every properly infinite projection is $\aleph_0$-homogeneous.
Therefore $q$ is $\aleph_0$-homogeneous and hence equivalent to its central support projection $c_q$.
Let $\omega$ be the normal tracial weight on $M$ which is zero on $Mc_q$ and equal to $\infty$ on every positive element that does not lie in $Mc_q$. Then $\omega(p) \le \omega(q) = 0$, which shows that $p\in Mc_q$, and hence $c_p \leq c_q$.
It now follows that $p\leq c_p \leq c_q \sim q$, and so $p\precsim q$ as desired.
\end{proof}

\noindent We can now show that Murray-von Neumann \mbox{(sub-)}equivalence of open projections in the bidual of a \Cs{} is equivalent to tracial \mbox{(sub-)}equivalence of the corresponding positive elements in the \Cs.

\begin{theor}
\label{prop:LDF-comp_implies_MvN-comp}
    Let $a$ and $b$ be positive elements in a separable \Cs{} $A$. Then $p_a \precsim p_b$ in $A^{**}$ if and only if $a \precsim_{\tr} b$ in $A$; and $p_a \sim p_b$  in $A^{**}$ if and only if $a \sim_{\tr} b$ in $A$.
\end{theor}

\begin{proof}
    The "only if parts" have already been proved in Corollary \ref{prop:MvN-comp_implies_LDF-comp}.
    Suppose that $a \precsim_{\tr} b$. Let $\omega$ be a normal tracial weight on $A^{**}$, and denote by $\omega_0$ its restriction to $A$.
    Then $\omega_0$ is a norm lower semicontinuous tracial weight on $A$, whence
    $$\omega(p_a) = d_{\omega_0}(a) \le d_{\omega_0}(b) = \omega(p_b),$$
    cf.\ Remark \ref{rem:traces}.
    As $\omega$ was arbitrary we can now conclude from Lemma \ref{prop:sep_C-alg_aleph_zero} and Proposition \ref{prop:comparison_proj_by_traces} that $p_a \precsim p_b$.

    The second part of the theorem follows easily from the first part.
\end{proof}

\begin{corol}
\label{prop:PZ_implies_Cu_implies_MvN}
    Let $A$ be a separable \Cs, and $p$ and $q$ be two open projections in $A^{**}$.
    Then:
    \begin{align*}
        p \precsim_{\PZ} q  \implies p \precsim_{\Cu} q \implies p \precsim q, \qquad  p \sim_{\PZ} q  \implies p \sim_{\Cu} q \implies p \sim q
    \end{align*}
The first implication in each of the two strings holds without assuming $A$ to be separable.
\end{corol}

\begin{proof}  Since $A$ is separable there are positive elements $a$ and $b$ such that $p = p_a$ and $q = p_b$. The corollary now follows from Remark \ref{remark:comparison}, Proposition \ref{prop:TFAE:equivalence}, Proposition \ref{prop:TFAE:Cuntz_subequivalence}, and
Theorem \ref{prop:LDF-comp_implies_MvN-comp}.
\end{proof}

\noindent
    It should be remarked, that one can prove the corollary above more directly without invoking Remark \ref{remark:comparison}.

\begin{rema}
    There is a certain similarity of our main results with the following result recently obtained by Robert in \cite[Theorem 1]{Rob}:
    If $a,b$ are positive elements of a \Cs{} $A$, then the following are equivalent:
    \begin{enumerate}
        \item
        $\tau(a)=\tau(b)$ for all norm lower semicontinuous tracial weights on $A$,
        \item
        $a$ and $b$ are Cuntz-Pedersen equivalent, i.e., there exists a sequence $\{x_k\}$ in $A$ such that  $a=\sum_{k=1}^\infty x_kx_k^*$ and $b=\sum_{k=1}^\infty x_k^*x_k$ (the sums are norm-convergent).
    \end{enumerate}

    It is known that Cuntz-Pedersen equivalence and Murray-von Neumann equivalence agree for projections in a \emph{von Neumann algebra} (see \cite[Theorem 4.1]{KadPed}).
\end{rema}

%########################################################################################################################
%########################################################################################################################
\section{Summary and applications.}
\label{sect:summary}

\noindent
In the previous sections we have established equivalences and implications between different types of comparison of positive elements and their corresponding open projections and Hilbert modules. The results we have obtained can be summarized as follows.  Given two positive elements $a$ and $b$ in a (separable) \Cs{} $A$ with corresponding open projections $p_a$ and $p_b$ in $A^{**}$ and Hilbert $A$-modules $E_a$ and $E_b$, then:
 \begin{equation} \tag{$\ast$} \label{eq:1}
\begin{split}
        \xymatrix@M=5pt{
        a\precsim_s b 	\ar@{<=>}[r] \ar@{=>}[d]
            & 	p_a\precsim_{\PZ}p_b \ar@{=>}[d] \\
        a\precsim b 	\ar@{<=>}[r] \ar@{=>}[d]
            & 	p_a\precsim_{\Cu}p_b \ar@{=>}[d] \\
        a\precsim_{\tr} b 	\ar@{<=>}[r] & p_a\precsim p_b
        }  \qquad \qquad
     \xymatrix@M=5pt{
        a\sim_s b 	\ar@{<=>}[r] \ar@{=>}[d]
            & 	p_a\sim_{\PZ}p_b \ar@{<=>}[r] \ar@{=>}[d]
            & 	E_a\cong E_b \ar@{=>}[d] \\
        a\approx b 	\ar@{<=>}[r] \ar@{=>}[d]
            & 	p_a\sim_{\Cu}p_b \ar@{<=>}[r] \ar@{=>}[d]
            & 	E_a\sim_{\Cu}E_b \\
        a\sim_{\tr} b 	\ar@{<=>}[r] & p_a\sim p_b
        }
\end{split}
 \end{equation}

\noindent
    We shall discuss below to what extend the reverse (upwards) implications hold. First we remark how the middle bi-implications yield an isomorphism between the Cuntz semigroup and a semigroup of open projections modulo Cuntz equivalence.

\begin{pargr} [The semigroup of open projections] \label{cu=p}
    Given a \Cs{} $A$. We wish to show that its Cuntz semigroup $\Cu(A)$ can be identified with an ordered semigroup of open projections in $(A \otimes \K)^{**}$. More specifically, we show
$\Proj_{\mathrm{o}}((A\otimes \K)^{**})/\!\!\sim_{\Cu}$ is an ordered abelian semigroup which is isomorphic to $\Cu(A)$.

    First we note how addition is defined on  the set
$\Proj_{\mathrm{o}}((A\otimes \K)^{**})/\!\!\sim_{\Cu}$. Note that
    $$A \otimes \B(\ell^2) \subseteq \cM(A \otimes \K) \subseteq (A \otimes \K)^{**}.$$
    Choose two isometries $s_1$ and $s_2$  in $\B(\ell^2)$ satisfying the Cuntz relation  $1=s_1s_1^*+s_2s_2^*$, and consider the isometries $t_1=1\otimes s_1$ and $t_2=1\otimes s_2$ in $\cM(A \otimes \K) \subseteq (A\otimes\K)^{**}$.
    For every positive element $a$ in $A \otimes \K$ and for every isometry $t$ in $\cM(A \otimes \K)$  we have $a \sim_s t a t^*$ in $A \otimes \K$ and $p_a \sim_{\PZ} t p_a t^* = p_{t at^*}$ in $(A \otimes \K)^{**}$.
    We can therefore define addition in $\Proj_{\mathrm{o}}((A\otimes \K)^{**})/\!\!\sim_{\Cu}$ by
    \begin{equation} \tag{$\ast \ast$} \label{eq:2}
        [p]_{\Cu}+[q]_{\Cu}   :=[t_1 p t_1^*+t_2 q t_2^*]_{\Cu}, \qquad p, q \in \Proj_{\mathrm{o}}((A\otimes \K)^{**}).
    \end{equation}
    The relation $\precsim_{\Cu}$ yields an order relation on $\Proj_{\mathrm{o}}((A\otimes \K)^{**})/\!\!\sim_{\Cu}$, which thus becomes an ordered abelian semigroup.

    Proposition \ref{prop:TFAE:Cuntz_subequivalence} and Corollary \ref{cor:TFAE:Cuntz_equivalence} applied to the \Cs{} $A \otimes \K$ yield that the mapping $\langle a \rangle \mapsto [p_a]_{\Cu}$, for $a \in (A \otimes \K)^+$, defines an isomorphism
    $$\Cu(A ) \cong \Proj_{\mathrm{o}}((A\otimes \K)^{**})/\!\!\sim_{\Cu}$$
    of ordered abelian semigroups whenever $A$ is a separable \Cs.
   In more detail, Proposition \ref{prop:TFAE:Cuntz_subequivalence} and Corollary \ref{cor:TFAE:Cuntz_equivalence} imply that the map $\langle a \rangle \mapsto [p_a]_{\Cu}$ is well-defined, injective, and order preserving. Surjectivity follows from the assumption that $A$ (and hence $A \otimes \K$) are separable, whence all open projections in $(A \otimes \K)^{**}$ are of the form $p_a$ for some positive element $a \in A \otimes \K$. Additivity of the map follows from the definition of addition defined in \eqref{eq:2} above and the fact that $\langle a \rangle + \langle b \rangle = \langle t_1at_1^* + t_2bt_2^* \rangle$ in $\Cu(A)$.
\end{pargr}

\begin{pargr}[The stable rank one case]
\label{pargr:sr1}
    It was shown by Coward, Elliott, and Ivanescu in \cite[Theorem 3]{CowEllIva} that in the case when $A$ is a separable \Cs{} with stable rank one, then two Hilbert $A$-modules are isometric isomorphic if and only if they are Cuntz equivalent, and that the order structure given by Cuntz subequivalence is equivalent to the one generated by inclusion of Hilbert modules together with isometric isomorphism (see also \cite[Theorem 4.29]{AraPerToms}).
    Combining those results with Proposition \ref{prop:TFAE:equivalence}, Proposition \ref{prop:TFAE:subequivalence}, Proposition \ref{prop:TFAE:Cuntz_subequivalence}, and Corollary \ref{cor:TFAE:Cuntz_equivalence} shows that the following holds for all $a,b\in A^+$ and for all $p,q \in \Proj_{\mathrm{o}}(A^{**})$:
    \begin{enumerate}
        \item[(1) \ ]
        $a \precsim b  \Leftrightarrow a \precsim_{s} b$, and $a \approx b \Leftrightarrow a \sim_{s} b$.
        \item[(1)'\ ]
        $p \precsim_{\Cu} q \Leftrightarrow p \precsim_{\PZ} q$, and $p \sim_{\Cu} q \Leftrightarrow p \sim_{\PZ} q$.
        \item[(2) \ ]
        If $a\precsim_s b$ and $b \precsim_s a$, then $a\sim_s b$.
        \item[(2)'\ ]
        If $p \precsim_{\PZ} q$ and $q \precsim_{\PZ} p$, then $p \sim_{\PZ}  q$.
    \end{enumerate}
    Hence the vertical implications between the first and the second row of \eqref{eq:1} can be reversed when $A$ is separable and of stable rank one.

 The right-implications in (1) and (2) (and hence in (1)' and (2)') above do not hold in general.
    Counterexamples were given by Lin in \cite[Theorem 9]{Lin1990}, by Perera in \cite[Before Corollary 2.4]{Per}, and by Brown and Ciuperca in \cite[Section 4]{BrownCiup}. For one such example take
 non-zero projections $p$ and $q$ in a simple, purely infinite \Cs. Then, automatically, $p \precsim q$, $p \precsim_s q$, $q \precsim_s p$, and $p \approx q$; but $p \sim q$ and $p \sim_s q$ hold (if and) only if $p$ and $q$ define the same $K_0$-class (which they do not always do).

%    Note that given projections $p,q$ in a \Cs{} $A$ they are Murray-von Neumann equivalence $p\sim q$ in % $A$ if and only if $p\sim_s q$ (Blackadar equivalence).
%    Furthermore,  they are Murray-von Neumann subequivalent $p\precsim q$ in $A$ if and only if
% $p\precsim_s q$, which in turn happens if and only if $p\precsim q$ in the sense of Cuntz comparison
%  (thus, the potential conflict of notation between Cuntz- and Murray-von Neumann comparison is
% dispelled).

%  A counterexample for the stably finite case was given by .
    It is unknown whether (1)--(2)' hold for residually stably finite \Cs s, and in particular whether they 
hold for stably finite simple \Cs s.
\end{pargr}

\begin{pargr}[Almost unperforated Cuntz semigroup]
    We discuss here when the vertical implications between the second and the third row of \eqref{eq:1} can be reversed.
    This requires both a rather restrictive assumption on the \Cs{} $A$, and also an assumption on the positive elements $a$ and $b$.
    To define the latter, we remind the reader of the notion of purely non-compact elements from \cite[Before Proposition 6.4]{EllRobSan}:
    The quotient map $\pi_I \colon A \to A/I$ induces a morphism $\Cu(A) \to \Cu(A/I)$ whenever $I$ is an ideal in $A$. 
    An element $\langle a\rangle$  in $\Cu(A)$ is \emph{purely non-compact} if whenever $\langle \pi_I(a)\rangle$ is compact for some ideal $I$, it is properly infinite, i.e., $2\langle \pi_I(a)\rangle=\langle \pi_I(a)\rangle$ in $\Cu(A/I)$.
   Recall that an element $\alpha$ in the Cuntz semigroup $\Cu(B)$ of a \Cs{} $B$ is called \emph{compact} if it is \emph{way-below} itself, i.e., $\alpha\ll \alpha$ (see the end of Section \ref{section_comp_posi} for the definition).
%    The existence of suprema in the Cuntz semigroup of a \emph{stable} \Cs{} was shown by Coward, Elliott, and Ivanescu in \cite{CowEllIva}.

%    It is proved in \cite[Proposition 6.4]{EllRobSan} that the set of purely non-compact elements of $\Cu(A)$ forms a subsemigroup closed under the passage to suprema of increasing sequences.

It is shown in \cite[Theorem 6.6]{EllRobSan} that if $\Cu(A)$ is almost unperforated and if $a$ and $b$ are positive elements in $A \otimes \K$ such that $\langle a \rangle$ is purely non-compact in $\Cu(A)$, then $\widehat{\langle a \rangle} \le \widehat{\langle b \rangle}$ implies that $\langle a \rangle \le \langle b \rangle$ in $\Cu(A)$. In the notation of \cite{EllRobSan}, and using \cite[Proposition 4.2]{EllRobSan}, $\widehat{\langle a \rangle} \le \widehat{\langle b \rangle}$ means that $d_\tau(a) \le d_\tau(b)$ for every (lower semicontinuous, possibly unbounded) 2-quasitrace on $A$. In the case where $A$ is exact it is known that all such 2-quasitraces are traces by Haagerup's theorem, \cite{Haa92}, (extended to the non-unital case by Kirchberg, \cite{Kir97}), and so it follows that $\widehat{\langle a \rangle} \le \widehat{\langle b \rangle}$ if and only if $a \precsim_{\tr} b$. We can thus rephrase \cite[Theorem 6.6]{EllRobSan} (see also \cite[Corollary 4.6 and Corollary 4.7]{Ror2004}) as follows:
    Suppose that $A$ is an exact, separable \Cs{} with $\Cu(A)$ almost unperforated.
    Then the following holds for all positive elements $a,b$ in $A \otimes \K$ :
    \begin{enumerate}
        \item[(3) \ ]
        If $\langle a \rangle \in \Cu(A)$ is purely non-compact, then $a \precsim_{\tr} b \Leftrightarrow a\precsim b$.
        \item[(4) \ ]
        If $\langle a \rangle, \langle b \rangle \in \Cu(A)$ are purely non-compact, then $a \sim_{\tr} b \Leftrightarrow a\approx b$.
    \end{enumerate}

  We wish to rephrase (3) and (4) above for open projections. We must first deal with the problem of choosing which kind of compactness of open projection to be invoked.
  Compactness of an open projection $p\in A^{**}$  as in Definition \ref{defn:cpct_containment_opnProj}  means that $p\in A$ (see Proposition \ref{prop:cpct_opn_proj_lies_in_A}).
    On the other hand, compactness for an element of the Cuntz semigroup $\Cu(A)$ is defined in terms of its ordering. Compactness of $p_a$ implies compactness of $\langle a \rangle\in\Cu(A)$ for every positive element $a$ in $A \otimes \K$.
    Brown and Ciuperca have shown that the converse holds in stably finite \Cs s, \cite[Corollary 3.3]{BrownCiup}.
    Recall that a \Cs{} is called stably finite if its stabilization contains no infinite projections.

    From now on, we restrict our attention to the residually stably finite case, which means that all quotients of the \Cs{} are stably finite.
    We define an open projection $p$ in $A^{**}$ to be \emph{residually non-compact} if there is no closed, central projection $z\in A^{**}$ such that $p z$ is a non-zero, compact (open) projection in $A^{**}z$.
    Here, we identify $A^{**}z$ with the bidual of the quotient $A/I$, where $I$ is the ideal corresponding to the open, central projection $1-z$, i.e., $I=A_{1-z}=(1-z)A^{**}(1-z)\cap A$.

 It follows from Proposition \ref{prop:cpct_opn_proj_lies_in_A} that an open projection $p\in A^{**}$ is residually non-compact if and only if there is no closed, central projection $z\in A^{**}$ such that $p z$ is non-zero and belongs to $Az$.
    Applying \cite[Corollary 3.3]{BrownCiup} to each quotient of $A$, we get that $\langle a \rangle\in\Cu(A)$ is purely non-compact if and only if $p_a$ is residually non-compact whenever $a$ is a positive element in $A \otimes \K$.

    Thus, for open projections $p,q$ in the bidual of a separable, exact, residually stably finite \Cs{} $A$ with $\Cu(A)$ almost unperforated, the following hold:
    \begin{enumerate}
        \item[(3)'\ ]
        If $p$ is residually non-compact, then $p\precsim q \Leftrightarrow p\precsim_{\Cu} q$.
        \item[(4)'\ ]
        If $p$ and $q$ are residually non-compact, then $p \sim q \Leftrightarrow p\sim_{\Cu} q$.
    \end{enumerate}

 If, in addition, $A$ is assumed to be simple, then an open projection $p$ in $A^{**}$ is residually non-compact if and only if it is not compact, i.e., if and only if $p\notin A$, thus:
    \begin{enumerate}
        \item[(3)''\ ]
        If $p\notin A$, then $p\precsim q \Leftrightarrow p\precsim_{\Cu} q$.
        \item[(4)''\ ]
        If $p,q\notin A$, then $p \sim q \Leftrightarrow p\sim_{\Cu} q$.
    \end{enumerate}

    If $A$ is stably finite, and $p,q$ are two Cuntz equivalent open projections in $A^{**}$, then $p$ is compact if and only if $q$ is compact (see \cite[Corollary 3.4]{BrownCiup}).
    Together with (3)'' and (4)'' this gives the following new picture of the Cuntz semigroup:
    Let $A$ be a separable, simple, exact, stably finite \Cs{} with $\Cu(A)$ almost unperforated.
    Then
    $$ \Cu(A) = V(A)\ \coprod\ \big( \Proj_{\mathrm{o}}((A\otimes \K)^{**}) \setminus \! \Proj(A\otimes \K ) \big) /\! \! \sim.$$
    In other words, the Cuntz semigroup can decomposed into the monoid $V(A)$ (of Murray-von Neumann equivalence classes of projections in $A \otimes \K$) and the non-compact open projections modulo Murray-von Neumann equivalence in $(A \otimes \K)^{**}$.

In conclusion, let us note that the vertical implications between the second and the third row of \eqref{eq:1} cannot be reversed in general. Actually, these implications will fail whenever $\Cu(A)$ is not almost unperforated, which tends to happen when $A$ has "high dimension". These implications can also fail for projections in very nice \Cs s. Indeed, if $p$ and $q$ are projections, then $p\sim_{\tr}q$ simply means that $\tau(p)=\tau(q)$ for all traces $\tau$. It is well-known that the latter does not imply Murray-von Neumann or Cuntz equivalence even for simple AF-algebras, if their $K_0$ groups have non-zero infinitesimal elements.
\end{pargr}

\vspace{.5cm}

\section*{Acknowledgments}
We thank Uffe Haagerup for his valuable comments on von Neumann algebras that helped us to shorten and improve some of the proofs in Section \ref{sec5}.

\bibliographystyle{annotate}
\bibliography{References}

\end{document}